\documentclass[]{amsart}
\usepackage{mathrsfs}
\usepackage{graphicx}
\usepackage{amssymb}
\usepackage{amsmath}
\usepackage{amsthm}
\usepackage{galois}
\usepackage[all,cmtip]{xy}

\allowdisplaybreaks                 
\usepackage{float}                 
\restylefloat{figure}              
\usepackage{caption2}              
\usepackage{array}              
\usepackage{slashbox}              
\newtheorem{thm}{Theorem}[section]

\newtheorem{lem}[thm]{Lemma}
\newtheorem{prop}[thm]{Proposition}
\newtheorem{claim}[thm]{Claim}
\theoremstyle{definition}
\newtheorem{defn}[thm]{Definition}
\newtheorem{exam}[thm]{Example}
\newtheorem{ques}[thm]{Question}
\theoremstyle{remark}
\newtheorem{rem}[thm]{Remark}
\numberwithin{equation}{section}

\newcommand{\Z}{\mathbb Z}

\newcommand{\fix}{\mathrm{Fix}}
\newcommand{\ind}{\mathrm{ind}}

\newcommand{\rank}{\mathrm{rank}}
\newcommand{\stab}{\mathrm{Stab}}
\newcommand{\im}{\mathrm{Im}}
\newcommand{\F}{\mathbf{F}}      

\begin{document}

\title{The fixed subgroups of homeomorphisms of Seifert manifolds}\author{Qiang Zhang}\address{School of Mathematics and Statistics, Xi'an Jiaotong University,
Xi'an 710049, China}\email{zhangq.math@mail.xjtu.edu.cn}




\begin{abstract}
Let $M$ be a compact connected orientable Seifert manifold with hyperbolic orbifold $B_M$, and  $f_{\pi}: \pi_1(M)\rightarrow\pi_1(M)$ be an automorphism induced by an orientation-reversing homeomorphism $f$ of $M$. We give a bound on the rank of the fixed subgroup of $f_{\pi}$, namely, $\rank\fix(f_{\pi})<2\rank \pi_1(M)$, which is similar to the inequalities on surface groups and hyperbolic 3-manifold groups.
\end{abstract}
\maketitle
\section{Introduction}

For a group $G$ and an endomorphism $\phi: G\rightarrow G$, the $fixed$ $subgroup$ of $\phi$ is
$$\fix(\phi):=\{g\in G|\phi(g)=g\},$$
which is a subgroup of $G$. Let $\rank G$ denote the $rank$ of $G$, which means the minimal number of the generators of $G$.

For a free group $F$ and an automorphism $\phi$, M. Bestvina and M. Handel \cite{BH} solved the Scott conjecture:

\begin{thm}[Bestvina-Handel]
Let $F$ be a free group and $\phi$ be an automorphism of $F$. Then
$$\rank\fix(\phi)\leq \rank F.$$
\end{thm}

In the paper by B. Jiang, S. Wang and Q. Zhang \cite{JWZ}, it is proved that

\begin{thm}[Jiang-Wang-Zhang]\label{single endomorphism on surface gp}
Let $S$ be a compact surface and $\phi$ be an endomorphism of $\pi_1(S)$. Then
$$\rank\fix(\phi)\leq \rank \pi_1(S).$$
\end{thm}

In a recent paper \cite{WZ}, J. Wu and Q. Zhang generalized Theorem \ref{single endomorphism on surface gp} to a family endomorphisms of a surface group.
In \cite{LW} (see \cite{Z2} for an enhance version), J. Lin and S. Wang showed that

\begin{thm}[Lin-Wang]
Let $M$ be a compact orientable hyperbolic 3-manifold with finite volume and $\phi$ be an automorphism of $\pi_1(M)$. Then
$$\rank\fix(\phi)<2\rank \pi_1(M).$$
\end{thm}

In this paper, we consider the fixed subgroups of automorphisms of Seifert 3-manifold groups.

Suppose $M$ is a compact orientable 3-manifold. We say that $M$ is a $Seifert$ $manifold$, if $M$ possesses a $Seifert$ $fibration$ which is a decomposition of $M$ into disjoint simple closed curves, called $fibers$, such that each fiber has a solid torus neighborhood consisting of a union of fibers. Identifying each fiber of $M$ to a point, we get a set $B_M$, called the $orbifold$ of $M$, which has a natural 2-orbifold structure with singular points consisting of cone points. It is useful to think of a Seifert manifold as a circle bundle over a 2-orbifold. For brevity, we say an orbifold it means a compact 2-orbifold with singular points consisting of cone points in the following. An orbifold (or surface) is called $hyperbolic$ if it has negative Euler characteristics. A hyperbolic orbifold is orbifold covered by a hyperbolic surface and admits a hyperbolic structure with totally geodesic boundary. For more information about orbifolds, see \cite[\S1 and \S2]{JWW} and \cite[\S2]{S}. A map $f$ on a Seifert manifold $M$ is called $fiber$-$preserving$ if it maps fibers to fibers. If $f$ is fiber-preserving, then it
induces a map $f': B_M\rightarrow B_M$ on the orbifold $B_M$.

In the following, all spaces are assumed to be connected and compact unless it is specially stated otherwise. For a set $X$, let $\#X$ denote the number of points in $X$.

The main result of this paper is

\begin{thm}\label{main thm}
Suppose $M$ is a compact connected orientable Seifert manifold (closed or with toroidal boundary) with hyperbolic orbifold $B_M$, and  $f_{\pi}: \pi_1(M)\rightarrow\pi_1(M)$ is an automorphism induced by an orientation-reversing homeomorphism $f: M\rightarrow M$. Then
$$\rank\fix(f_{\pi})<2\rank \pi_1(M).$$
\end{thm}

\begin{rem}
By the well known Geometrization Theorem, $M$ is a Seifert manifold with hyperbolic orbifold $B_M$ if and only if $M$ admits a geometric structure based on one of the two geometries : $\mathbb H^2\times \mathbb R$, $\widetilde{SL(2,\mathbb R)}$. The condition that $f$ is orientation-reversing is necessary. If $f$ is orientation-preserving, then the fixed subgroup $\fix(f_{\pi})$ can be infinitely generated, see \cite[Example 5.2]{Z}.
\end{rem}

\begin{rem}
By \cite[Example 5.1]{Z}, one can show that there is an orientation-reversing homeomorphism $f$ of $M_n=S_n\times S^1$, where $S_n$ is a closed orientable surface of genus $n\geq 2$, such that $\rank\fix(f_{\pi})=4n-2$.
Therefore, for any $\varepsilon>0$, there exists a Seifert manifold $M_n$ and an orientation-reversing homeomorphism $f$ of $M_n$, such that
$$\frac{\rank\fix(f_{\pi})}{\rank\pi_1(M_n)}=\frac{4n-2}{2n+1}>2-\varepsilon.$$
\end{rem}

Theorem \ref{main thm} is inspired by the following proposition \cite[Corollary 1.4]{Z}.

\begin{prop}\label{rank of essential fpc}
Suppose $f: M\rightarrow M$ is a homeomorphism of a compact connected orientable Seifert manifold with hyperbolic orbifold $B_M$. Let $f_*: \pi_1(M,x)\rightarrow \pi_1(M,x)$ be the induced automorphism, where $x$ is a fixed point contained in an essential fixed point class of $f$. Then
$$\rank\fix(f_*)<2\rank \pi_1(M).$$
\end{prop}
The paper is organized as follows. In Section 2, we will give some background on fixed points and fixed subgroups of a selfmap. In Section 3, we will give some useful facts on Fuchsian groups. In Section 4, we will consider the special case of Theorem \ref{main thm} that $M$ is an orientable circle bundle
over an orientable hyperbolic surface. Finally, we will finish the proof of Theorem \ref{main thm} in Section 5, and give some examples and questions in Section 6.\\

\noindent\textbf{Acknowledgements.} The author would like to thank Professor Boju Jiang for valuable communications. This work was carried out while
the author was visiting Princeton University and he would like to thank for their hospitality. The author is partially supported by NSFC (No. 11201364) and ``the Fundamental Research Funds for the Central Universities".\\

\section{Fixed points and fixed subgroups}

Let $X$ be a connected compact polyhedron and $f:X\rightarrow X$ be a selfmap. In this section, we introduce some facts on fixed point classes and fixed subgroups of $f$.

The fixed point set
$$\fix f:=\{x\in X|f(x)=x\}$$
splits into a disjoint union of $fixed~ point~ classes$: two fixed points are in the same class if and only if they can be joined by a $Nielsen~ path$, which is a path homotopic (rel. endpoints) to its own $f$-image. For each fixed point class $\F$, a homotopy invariant $index$ $\ind(f,\F)\in \Z$ is defined. A fixed point class is $essential$ if its index is non-zero, otherwise, called $inessential$ (see \cite{fp1} for an introduction).

Although there are several approaches to define fixed point classes, we state the one using paths and introduce another homotopy invariant $rank$ $\rank(f,\F)\in \Z$ for each fixed point class $\F$ (see \cite[\S2]{JWZ}), which plays a key role in this paper.

\begin{defn}
By an $f$-$route$ we mean a homotopy class (rel. endpoints) of path $w:I\rightarrow X$ from a point $x\in X$ to $f(x)$. For brevity we shall often say the path $w$ (in place of the path class $[ w ]$) is an $f$-route at $x=w(0)$. An $f$-route $w$ gives rise to an endomorphism
$$f_{w}:\pi_1(X,x) \rightarrow\pi_1(X,x),~~[ a] \mapsto [ w(f\comp a)\overline w ] $$
where $a$ is any loop based at $x$, and $\overline w$ denotes the reverse of $w$. For brevity, we will write
$$f_{\pi}: \pi_1(X)\to\pi_1(X)$$
when $w$ and the base point $x$ are omitted.

Two $f$-routes $[ w ]$ and $[ w' ]$ are $conjugate$ if there is a path $q:I\rightarrow X$ from $x=w(0)$ to $x'=w'(0)$ such that $[ w' ]=[ \overline qw (f\comp q)]$, that is $w'$ and $\overline qw (f\comp q)$ homotopic rel. endpoints. We also say that the (possibly tightened) $f$-route $\overline qw (f\comp q)$ is obtained from $w$ by an $f$-$route$ $move$ along the path $q$.
\end{defn}

Note that a constant $f$-route $w$ corresponds to a fixed point $x=w(0)=w(1)$ of $f$, and the endomorphism $f_{w}$ becomes the usual
$$f_*:\pi_1(X,x)\rightarrow \pi_1(X,x),~~[ a]\mapsto  [ f\comp a],$$
where $a$ is any loop based at $x$. Two constant $f$-routes are conjugate if and only if the corresponding fixed points can be joint by a Nielsen path. This gives the following definition.

\begin{defn}
With an $f$-route $w$ (more precisely, with its conjugacy class) we associate a $fixed$ $point$ $class$ $\F_{w}$ of $f$, which consists of the fixed points that correspond to constant $f$-routes conjugate to $w$. Thus fixed point classes are associated bijectively with conjugacy classes of $f$-routes. A fixed point class $\F_{w}$ can be empty if there is no constant $f$-route conjugate to $w$. Empty fixed point classes are inessential and distinguished by their associated route conjugacy classes.
\end{defn}

\begin{defn}
The $fixed$ $subgroup$ of the endomorphism $f_w$ is the subgroup
$$\fix(f_w):=\{\gamma\in \pi_1(X,w(0))|f_w(\gamma)=\gamma\}.$$
The $stabilizer$ of the fixed point class $\F_w$ is defined to be
$$\stab(f,\F_w):=\fix(f_w),$$
it is well defined up to isomorphism because conjugate $f$-routes have isomorphic stabilizers. Hence, we have the $rank$ of $\F_w$ defined as
$$\rank(f,\F_w):=\rank \stab(f,\F_w)=\rank\fix(f_w).$$
\end{defn}

The following are some facts on stabilizer (see \cite[\S2]{JWZ}).

\textbf{Fact} (Homotopy invariance). A homotopy $H=\{h_t\}_{t\in I}:X\rightarrow X$ gives rise to a bijective correspondence $H: \F_{w_0}\mapsto \F_{w_1}$ from $h_0$-fixed point classes to $h_1$-fixed point classes, and
$$\ind(h_0,\F_{w_0})=\ind(h_1, \F_{w_1}), ~~\stab(h_0,\F_{w_0})\cong \stab(h_1, \F_{w_1}),$$
which indicates that the index $\ind(f,\F_w)$ and the rank $\rank(f,\F_w)$ of a fixed point class are both homotopy invariants.

\textbf{Fact} (Morphism). A morphism from a selfmap $f: X\rightarrow X$ to a selfmap $g:Y\rightarrow Y$ means a map $h:X\rightarrow Y$ such that $h\comp f=g\comp h$. It induces a natural function $w\mapsto h\comp w $ from $f$-routes to $g$-routes and a function $\F_{w}\mapsto \F_{h\comp w}$  from $f$-fixed point classes to $g$-fixed point classes, such that
$$h(\F_w)\subseteq \F_{h\comp w},~~~h_*\stab(f,\F_{w})\leq\stab(g,\F_{h\comp w}).$$

The following are useful lemmas on covering spaces and fixed point classes.

\begin{lem}\label{fixed point class lifting}
Let $p:\widetilde M\rightarrow M$ be a finite covering of a compact manifold $M$, and $f:M\rightarrow M$ be a homeomorphism. Suppose $\tilde f:\widetilde M\rightarrow \widetilde M$ is a lifting of $f$, and the $\tilde f$-route $\tilde w$ is a lifting of the $f$-route $w$. Then the $f$-fixed point class $\F_w$ associated to $w$ is essential if and only if the $\tilde f$-fixed point class $\F_{\tilde w}$ associated to $\tilde w$ is essential, moreover,
$$\ind(\tilde f,\F_{\tilde w})=n\times \ind(f,\F_w)$$
where $n$ is a positive integer.
\end{lem}

\begin{proof}
Since indices of fixed point classes are homotopy invariants of the map $f$, via perturbation, we may assume that all fixed points of $f$ are isolated.

Firstly, we claim that $\F_{\tilde w}$ is nonempty if and only if $\F_w$ is nonempty, and $p(\F_{\tilde w})=\F_w$.

If $\F_{\tilde w}$ is nonempty, then for any point $\tilde x\in \F_{\tilde w}\subseteq\fix\tilde f$, the $\tilde f$-route $\tilde w$ is conjugate to $\tilde x$, namely, there is a path $\gamma: I\rightarrow\widetilde M$ from $\tilde w(0)$ to $\tilde x$ such that $\overline\gamma\tilde w(\tilde f\comp\gamma)\simeq \tilde x$ rel. $\tilde x$. Hence
$$(\overline {p\comp\gamma})w(f(p\comp\gamma))=p(\overline\gamma\tilde w(\tilde f\comp\gamma))\simeq p(\tilde x).$$
Namely, the $f$-route $w$ is conjugate to the point $p(\tilde x)\in\F_w\subseteq\fix f$. So $p(\F_{\tilde w})\subseteq\F_w$.

If $\F_w$ is nonempty, then for any point $x\in\F_w\subseteq\fix f$, the $f$-route $w$ is conjugate to $x$, namely, there is a path $c$ from $w(0)$ to $x$ such that $\overline cw(f\comp c)\simeq x$ rel. $x$. Let $\tilde c$ be a lifting of $c$ from $\tilde w(0)$ to a point $\tilde x\in p^{-1}(x)$. Then $\overline{\tilde c}\tilde w(\tilde f\comp\tilde c)$ is a lifting of the contractible loop $\overline cw(f\comp c)$. Hence $\overline{\tilde c}\tilde w(\tilde f\comp\tilde c)$ is also a contractible loop and $\tilde f(\tilde x)=\tilde x$, namely, the $\tilde f$-route $\tilde w$ is conjugate to the point $\tilde x\in\F_{\tilde w}\subseteq\fix \tilde f$. So $\F_w\subseteq p(\F_{\tilde w})$ and the claim holds.

Secondly, we prove $\ind(\tilde f,\F_{\tilde w})=n\times \ind(f,\F_w).$

If $\F_w$ is empty, then $\F_{\tilde w}$ is also empty according to the claim above, and the equation holds clearly.

Now we consider the case that $\F_w$ is nonempty. Pick a point $x\in \F_w$. We can assume
$$p^{-1}(x)\cap\F_{\tilde w}=\{\tilde x_1,\ldots,\tilde x_n\}$$
where $n>0$ since the covering $p:\widetilde M\rightarrow M$ is finite and $p(\F_{\tilde w})=\F_w$. If $y\neq x$ and $y\in \F_w$, then there is a Nielsen path $c$ from $x$ to $y$ such that $c\simeq f\comp c$ rel. $\{x,y\}$ by the definition of fixed point class. Hence there are $n$ liftings $\tilde c_i$ with $\tilde c_i(0)=\tilde x_i$ and $\tilde f\comp \tilde c_i\simeq \tilde c_i$ rel. endpoints, $i=1,\ldots,n$. Therefor
$$\{\tilde c_1(1),\ldots,\tilde c_n(1)\}=p^{-1}(y)\cap\F_{\tilde w}.$$
Then
$$\#\F_{\tilde w}=n\#\F_w.$$
So the equality holds by the fact $\ind(\tilde f,\tilde x)=\ind(f, x)$.
\end{proof}

A covering $p:\widetilde M\rightarrow M$ between compact manifolds is called $characteristic$ if the subgroup $p_*\pi_1(\widetilde M)$ is a characteristic subgroup of $\pi_1(M)$, i.e., it is invariant under any automorphism of $\pi_1(M)$. Recall that if $G$ is a finite
index subgroup of a finitely generated group $H$, then there is a
finite index characteristic subgroup $G'$ of $H$, such that
$G'\subseteq G$. It follows that given any finite covering
$p:\widetilde M\rightarrow M$ of a compact manifold $M$, there is a
finite covering $q:\widehat M\rightarrow \widetilde M$ so that
$p\comp q:\widehat M\rightarrow M$ is a characteristic covering. In this case for
any homeomorphism $f: M\rightarrow M$ with $y=f(x)$, and
any points $\tilde x$ and $\tilde y$ covering $x$ and $y$
respectively, there is a lifting $\tilde f$ of $f$ such that $\tilde
f(\tilde x)=\tilde y$. So we have the following

\begin{lem}\label{lifting of route}
Let $p:\widetilde M\rightarrow M$ be a characteristic covering between compact manifolds, and $\tilde w$ be a lifting of an $f$-route $w$ of a homeomorphism $f:M\rightarrow M$. Then there is a lifting $\tilde f$ of $f$ such that $\tilde w$ is an $\tilde f$-route.
\end{lem}

\section{Some facts on Fuchsian groups}

In this section, we give some facts on Fuchsian groups.

Let $B$ be an orbifold. Recall that an orbifold in this paper means a compact 2-orbifold with singular points consisting of cone points, then we can assume $B=F(n_1,n_2,\ldots, n_k)$, where the compact surface $F$ denotes the underlying space of $B$ and $n_i\geq 2$ denotes the cone point with cone angle $2\pi/n_i$, $i=1,2,\ldots,k$.  The fundament group $\pi_1(B)$ is called a $Fuchsian$ $group$. If $H$ is a subgroup of the Fuchsian group $\pi_1(B)$, then $H$ is also a Fuchsian group because there is a covering orbifold $\widetilde B$ of $B$ such that $H\cong \pi_1(\widetilde B)$. Moreover, if $\widetilde B$ is also compact, then $\chi(\widetilde B)=k\chi(B)$ and $H$ has finite index $k$ in $\pi_1(B)$, where $k>0$ is the degree of the covering. In particular, if the orbifold $B$ is hyperbolic, then $B$ can be covered by a hyperbolic surface $S$, and $\pi_1(B)$ is an infinite Fuchsian group since $\pi_1(B)$ has a subgroup isomorphic to the infinite group $\pi_1(S)$. For more information on Fuchsian groups, see \cite[Chapter 2]{JS}.

\begin{lem} \label{center of Fuchsian gp}
{\rm(1)} Any infinite Fuchsian group with nontrivial center is either free abelian of rank $\leq 2$ or isomorphic to the fundamental group of a Klein bottle.

{\rm(2)} Any finite Fuchsian group is either cyclic or isomorphic to
$$\langle a,b|a^{n_1}=b^{n_2}=(ab)^{n_3}=1\rangle,~~~\frac{1}{n_1}+\frac{1}{n_2}+\frac{1}{n_3}>1, n_1,n_2,n_3\geq 2,$$
which is the fundamental group of the closed orbifold $O=S^2(n_1,n_2,n_3)$ with $\chi(O)=\frac{1}{n_1}+\frac{1}{n_2}+\frac{1}{n_3}-1>0$.
\end{lem}

\begin{proof}
Conclusion (1) is from \cite[Proposition II.3.11]{JS}, and Conclusion (2) can be verified by \cite[Theorem 12.2]{He} clearly.
\end{proof}

For a group $G$ and an element $g\in G$, let $C_G(g)=\{x\in G|xg=gx\}$ be the centralize of $g$ in $G$, and $C(G)=\{x\in G|xg=gx, \forall g\in G\}$ the center of $G$. We have

\begin{lem}\label{centralize of an infinite-order element}
Let $B$ be a hyperbolic orbifold and $a\in \pi_1(B)$ be an element of infinite order. Then for any $i\neq 0$,
$$C_{\pi_1(B)}(a^i)=C_{\pi_1(B)}(a)\cong \Z.$$
\end{lem}

\begin{proof}
For any $i\neq 0$, $C_{\pi_1(B)}(a^i)\leq \pi_1(B)$ is an infinite Fuchsian group with nontrivial center since the infinite cyclic group $\langle a\rangle\leq C_{\pi_1(B)}(a^i)$. Then $C_{\pi_1(B)}(a^i)$ is either free abelian of rank $\leq 2$ or the fundamental group of a Klein bottle by Lemma \ref{center of Fuchsian gp}(1). Note that $B$ is hyperbolic and $C_{\pi_1(B)}(a^i)$ is a subgroup of $\pi_1(B)$, then $C_{\pi_1(B)}(a^i)$ is neither the fundamental group of a Klein bottle nor the fundamental group of a torus because neither of them is hyperbolic. Therefore, $C_{\pi_1(B)}(a^i)$ is free cyclic, set $C_{\pi_1(B)}(a^i)=\langle c\rangle\cong \Z$. Since $a\in C_{\pi_1(B)}(a^i)=\langle c \rangle$, $a$ is a power of $c$. Thus $\langle c\rangle \leq C_{\pi_1(B)}(a)$, i.e., $C_{\pi_1(B)}(a^i)\leq C_{\pi_1(B)}(a)$. Clearly,  $C_{\pi_1(B)}(a)\leq C_{\pi_1(B)}(a^i)$. Thus
$C_{\pi_1(B)}(a^i)=C_{\pi_1(B)}(a)\cong \Z$ for any $i\neq 0$.
\end{proof}

Now we give two lemmas which are used in the following.

In group theory, a group $G$ is called $metacyclic$ if it contains a cyclic, normal subgroup $N$ such that the quotient group $G/N$ is also cyclic. Clearly, the rank of a metacyclic group is no more than $2$. In particular, cyclic groups are metacyclic. For metacyclic groups, we have

\begin{lem}\label{subgp of metacyclic gp}
{\rm(1)} Any subgroup of a metacyclic group is also metacyclic.

{\rm(2)} Let $G$ be a group. If $G$ has a cyclic normal subgroup $T$ such that the quotient group $G/T$ is metacyclic, then any subgroup of $G$ has rank $\leq 3$.
\end{lem}

\begin{proof}
(1) Let $G$ be a metacyclic group, i.e., there is a cyclic normal subgroup $N$ such that $G/N$ is also cyclic. Let $H$ be any subgroup of $G$. Then $H\cap N$ is a cyclic normal subgroup of $H$, and $H/H\cap N\cong HN/N$ is also cyclic since $HN/N\leq G/N$. Thus $H$ is metacyclic.

(2) Let $H$ be any subgroup of $G$. Note that $H\cap T$ is a cyclic normal subgroup of $H$ because $T$ is cyclic normal in $G$. Then the quotient group $H/H\cap T\cong HT/T\leq G/T$ is metacyclic and $\rank (H/H\cap T)\leq 2$ according to conclusion (1). Thus $\rank H\leq 3$.
\end{proof}

\begin{lem}\label{infinite elements}
Suppose $G$ is a group and $c\in G$ is an element of infinite order. If all the infinite-order elements are contained in the infinite cyclic group $\langle c\rangle$, then $G$ is a metacyclic group. More precisely, either $G=\langle c\rangle\cong \Z$ or $G=\langle a,c|a^2=1, aca^{-1}=c^{-1}\rangle\cong \Z_2*\Z_2$.
\end{lem}

\begin{proof}
If there are no nontrivial finite-order elements in $G$, then $G=\langle c\rangle\cong \Z$ clearly.

Now we assume that there is an element $a\in G$ of finite order $n\geq 2$. Consider the subgroup $H=\langle a, c\rangle\leq G$. Clearly $aca^{-1}$ is also an element of infinite order, thus $aca^{-1}\in \langle c\rangle$ and $\langle c\rangle$ is a normal subgroup of $H$. Hence $aca^{-1}=c^i$ for some $i\neq 0$, and $H$ is a metacyclic group. Then
$$c=a^nca^{-n}=a^{n-1}(aca^{-1})a^{1-n}=a^{n-1}c^ia^{-n+1}=c^{i^n}.$$
Note that $c$ is of infinite order, we have $i^n=1$. Thus $i=1$, if $n$ is odd; $i=\pm 1$, if $n$ is even. Namely, $H$ is either

(1) $\langle a, c|a^n=1, aca^{-1}=c\rangle$, or

(2) $\langle a, c|a^n=1, aca^{-1}=c^{-1}\rangle, n ~{\rm even}$.

However, case (1) is impossible. In fact, note that $ac\in H$ is of infinite order, thus $ac\in \langle c\rangle$. Then $a\in \langle c\rangle$ which contradicts to the assumption that $a\in G$ has finite order $n\geq 2$.

In case (2), note that $a^2ca^{-2}=c$, we have an abelian subgroup $\langle a^2, c|a^n=1, a^2ca^{-2}=c\rangle$. It implies $a^2c$ has infinite order. Thus $a^2c\in \langle c\rangle$ and $a^2\in \langle c\rangle$. Recall that $a$ has finite order $n\geq 2$ and $c$ has infinite order, then $n=2$. It implies all the nontrivial finite-order elements in $G$ must have order $2$, and
$$H=\langle a, c|a^2=1, aca^{-1}=c^{-1}\rangle\cong \Z_2*\Z_2.$$

To prove $G=H$, it suffices to prove $G\leq H$. Suppose there is another element $b\in G$ of order $2$, then there is a subgroup $H'=\langle b, c|b^2=1, bcb^{-1}=c^{-1}\rangle$ according to the argument above. Thus $abcb^{-1}a^{-1}=c$, i.e., $ab$ commutes with $c$. If $abc$ has finite order $k>0$, then $(ab)^kc^k=(abc)^k=1$, i.e., $(ab)^k=c^{-k}$ has infinite order. Thus $ab$ has infinite order and $ab\in\langle c\rangle$, it implies $b\in \langle a, c\rangle=H$. If $abc$ has infinite order, then $abc\in\langle c\rangle$ which also implies $b\in \langle a, c\rangle=H$. Therefore, $G\leq H$ and the proof is finished.
\end{proof}

\begin{prop}\label{roots of cyclic subgp of Fuchsian gps}
Let $H$ be a subgroup of $\pi_1(B)$ where $B$ is a hyperbolic orbifold. If
$$H^d:=\{h^d|h\in H\}\subseteq \langle a \rangle\cong\Z$$
where $d\in \Z_+$ is a positive integer and $a\in \pi_1(B)$ of infinite order. Then $H$ is a metacyclic group, more precisely, $H$ is either cyclic or isomorphic to $\Z_2*\Z_2$.
\end{prop}

\begin{proof}
Note that $B$ is a hyperbolic orbifold, namely, it can be covered by a hyperbolic surface $S$. Then $\pi_1(B)$ is an infinite Fuchsian group since it has a subgroup isomorphic to the infinite group $\pi_1(S)$.

If $H^d=\{1\}$, then $H$ consists of elements of finite order. Thus $H$ is a finite Fuchsian group by the fact that every infinite Fuchsian group has an element of infinite order (see \cite[Lemma II.3.9]{JS}). Therefore, $H$ is either finite cyclic or isomorphic to the fundamental group $\pi_1(O)$ of a closed orbifold $O$ with $\chi(O)>0$ by Conclusion (2) of Lemma \ref{center of Fuchsian gp}. But the latter is impossible because $B$ with $\chi(B)<0$ can not be covered by $O$. Thus $H$ is finite cyclic.

If $H^d\neq \{1\}$, then $H$ contains some elements of infinite order. For any $h\in H$ of infinite order, there is an element $a^{i_h}\neq 1$ such that $h^d=a^{i_h}$ since $H^d\subseteq \langle a \rangle$. Then $h\in C_{\pi_1(B)}(a^{i_h})$. By Lemma \ref{centralize of an infinite-order element}, $C_{\pi_1(B)}(a^{i_h})=C_{\pi_1(B)}(a)=\langle c\rangle$ for an infinite-order element $c$, thus $h\in \langle c \rangle$. Then all the infinite-order elements of $H$ are contained in the infinite cyclic group $H\cap \langle c \rangle$. Therefore, $H$ is either infinite cyclic or isomorphic to $\Z_2*\Z_2$ by Lemma \ref{infinite elements}, and the proof is finished.
\end{proof}

\section{Homeomorphisms of circle bundles over surfaces}

In this section, let us consider the special case that $M$ is an orientable circle bundle over an orientable hyperbolic surface $S$ with fibration $p:M\to S$.

\begin{lem}\label{fiber oriented coherently}
Suppose $M$ is a compact orientable circle bundle over a compact orientable hyperbolic surface $S$. Then all fibers can be coherently oriented so that they present the same element of infinite order in the center of $\pi_1(M)$.
\end{lem}

\begin{proof}
This is clear from the presentation of $\pi_1(M)$ given in \cite[Chapter 12]{He}.
\end{proof}

Let $t$ be the element of $\pi_1(M)$ represented by the fiber. We call the infinite cyclic subgroup $\langle t \rangle$ the $fiber$ of $\pi_1(M)$ associated to the fibration $p$, and $f$ $preserves$ (resp. $reveres$) $the$ $fiber$ $orientation$ if $f_{\pi}(t)=t$ (resp. $f_{\pi}(t)=-t$).\\

Let $S$ be a compact connected hyperbolic surface. A standard form of homeomorphisms on $S$ is developed in \cite{JG} with fine-tuned local behavior from the Thurston canonical map from \cite{T}.\\

\noindent{\bfseries Theorem T.}\label{Thurston
thm} {\em Let $S$ be a compact connected hyperbolic surface. Every homeomorphism $f: S\rightarrow S$ is isotopic to a
homeomorphism $\varphi$ such that either

{\rm (1)} $\varphi$ is a periodic map, i.e., $\varphi^m=id$ for some
$m\geq 1$, or equivalently, $\varphi$ is an isometry with respect to some hyperbolic metric on $S$; or

{\rm (2)} $\varphi$ is a pseudo-Anosov map, i.e., there is a number
$\lambda>1$ and a pair of transverse measured foliations
$(\mathfrak{F}^s,\mu^s)$ and $(\mathfrak{F}^u,\mu^u)$ such that
$\varphi(\mathfrak{F}^s,\mu^s)=(\mathfrak{F}^s,\frac{1}{\lambda}\mu^s)$
and $\varphi(\mathfrak{F}^u,\mu^u)=(\mathfrak{F}^u,\lambda\mu^u)$;
or

{\rm (3)} $\varphi$ is a reducible map, i.e. there is a system of
disjoint simple closed curves $\Gamma=\{\Gamma_1,\cdots,\Gamma_k\}$
in $\mathrm{int}S$ with the property below.

~~~~ {\rm (a)} $\Gamma$ is invariant by $\varphi$ (but
the $\Gamma_i$'s may be permuted) and each
component of $S\backslash \Gamma$ has negative Euler characteristic.

~~~~ {\rm (b)} $\Gamma$ has a $\varphi$-invariant tubular neighborhood $\mathcal N(\Gamma)$ such that
on each (not necessarily connected) $\varphi$-component of
$S\backslash \mathcal N(\Gamma)$, $\varphi$ satisfies $(1)$ or $(2)$.

~~~~ {\rm (c)} $\Gamma$ is minimal among all systems satisfying {\rm (a)} and {\rm (b)}.

~~~~ {\rm (d)} $\varphi$ is in the {\rm  standard form} as defined in \cite[page 79]{JG}.}
\\

The $\varphi$ above will be called the \emph{standard form} isotopic to $S$.\\

\begin{lem}\label{fixed subgp of inessential fpc on surface}
Let $f:S\rightarrow S$ be a homeomorphism of an orientable compact hyperbolic surface $S$ and $\F_w$ be a fixed point class corresponding to an $f$-route $w$.
Then

{\rm(1)} If $\F_w$ is inessential, then $\rank\fix(f_w)\leq 1$;

{\rm(2)} If $f$ is orientation-reversing, then $\rank\fix(f_w)\leq 1$.
\end{lem}

\begin{proof}
Case (1) is clearly following from \cite[Theorem 1.1]{JWZ}.

Now we consider case (2). Since the index and rank of fixed point classes are both homotopy invariants, via an isotopy, we may assume that $f$ is in standard form. If the fixed point class $\F_w$ is empty, then case (2) holds according to case (1). If $\F_w$ is nonempty, it is connected since each Nielsen path of $S$ can be deformed (rel. endpoints) into $\fix f$ by \cite[Lemmas 1.2, 2.2 and 3.4]{JG}. Thus $\rank\fix(f_w)=\rank\pi_1(\F_w)$. Then case (2) can be proved by examining the list given in \cite[Lemma 3.6]{JG}.
\end{proof}

The lemma below is useful in the proof of Proposition \ref{fixed subgp of inessential fpc on circle bundles}.

\begin{lem}\label{fiber orientation-reversing homeo. preserves essential fpc}
Suppose $p:M\to S$ is a compact orientable circle bundle over a compact orientable hyperbolic surface $S$. Suppose $f:M\to M$ is a fiber-preserving homeomorphism that reverse the fiber orientation, and $f':S\to S$ is the induced homeomorphism. Let $w$ be an $f$-route. Then
the $f$-fixed point class $\F_w$ is essential if and only if the $f'$-fixed point class $\F_{p\comp w}$ is essential.
\end{lem}

\begin{proof}
Since the index of fixed point classes is a homotopy invariant, via a fiber-preserving homotopy, we can assume that the induced homeomorphism $f'$ is in standard form by \cite[Lemma 2.7]{JW}. Then $\ind(f,\F_w)$ equals either $\ind(f',\F_{p\comp w})$ or $2\ind(f',\F_{p\comp w})$ by \cite[Lemma 2.9]{JW}. Thus the conclusion holds.
\end{proof}

\begin{prop}\label{fixed subgp of inessential fpc on circle bundles}
Suppose $p:M\rightarrow S$ is a compact orientable circle bundle over an orientable hyperbolic surface $S$, $f:M\rightarrow M$
is a fiber-preserving homeomorphism that reveres the orientation of $M$, and $f': S\rightarrow S$ is the induced homeomorphism of $f$. Let $w$ be an $f$-route corresponding to an inessential fixed point class $\F_w$. Then

{\rm(1)} $\fix(f'_{p\comp w})$ is trivial or the free cyclic group $\mathbb{Z}$;

{\rm(2)} $\fix(f_w)$ is trivial or a free abelian group of rank $\leq 2$.
\end{prop}

\begin{proof}
(1) Note that all fibers can be coherently oriented since $M$ is a connected orientable circle bundle over an orientable surface by Lemma \ref{fiber oriented coherently}. So $f$ is in one of the two cases below.

Case (i). $f$ is fiber orientation-reversing and $f':S\rightarrow S$ is orientation-preserving.

Since the $f$-fixed point class $\F_w$ is inessential, the $f'$-fixed point class $\F_{p\comp w}$ is also inessential by Lemma \ref{fiber orientation-reversing homeo. preserves essential fpc}. Then $\rank\fix(f'_{p\comp w})\leq 1$ by Lemma \ref{fixed subgp of inessential fpc on surface}(1).

Case (ii). $f$ is fiber orientation-preserving and $f':S\rightarrow S$ is orientation-reversing. Then $\rank\fix(f'_{p\comp w})\leq 1$ by Lemma \ref{fixed subgp of inessential fpc on surface}(2).

Hence conclusion (1) holds in both case (i) and case (ii) by the fact that every nontrivial element of an orientable surface group has infinite order.\\

(2) Let $x=w(0)$. Since $p\comp f=f'\comp p: M\to S$, there is the following commutative diagram on fundamental groups:
$$\xymatrix{
\pi_1(M,x)\ar[d]_{p_*}\ar[r]^{f_w} &\pi_1(M,x)\ar[d]^{p_*}\\
\pi_1(S,p(x))\ar[r]^{f'_{p\comp w}} & \pi_1(S,p(x))
}$$
where
$$p_*:\pi_1(M,x)\to \pi_1(S,p(x))\cong\pi_1(M,x)/\langle t\rangle$$
is the quotient map and $\langle t\rangle$ is the fiber of $\pi_1(M,x)$ associated with the fibration $p$.
Hence
$$p_*\fix(f_w)\leq \fix(f'_{p\comp w}).$$
Since $\langle t\rangle$ is in the center of $\pi_1(M,x)$, and $\fix(f'_{p\comp w})$ is trivial or the infinite cyclic group $\Z$ by conclusion (1), we have
$$\fix(f_w)\leq p_*^{-1}\fix(f'_{p\comp w}) \cong\fix(f'_{p\comp w})\times \langle t\rangle,$$
which is isomorphic to the free abelian group $\Z$ or $\Z\oplus\Z$. Hence conclusion (2) holds by the fact that every nontrivial subgroup of the free abelian group $\Z\oplus\Z$ is a free abelian group of rank $\leq 2$.
\end{proof}
\section{Proof of Theorem \ref{main thm}}

In this section, we give the proof of Theorem \ref{main thm}.

Let $M$ be a compact orientable Seifert manifold with hyperbolic orbifold $B_M$, and $p:M\rightarrow B_M$ be the Seifert fibration. If $\tau$ is a regular fiber in this fibration, then
$$\langle t\rangle=\im(\pi_1(\tau,x)\to\pi_1(M,x))$$
is a infinite cyclic subgroup. We call $\langle t\rangle$ the $fiber$ of $\pi_1(M,x)$ (associated with the fibration $p$). In fact, for each Seifert fibration of $M$, there exists a unique fiber of $\pi_1(M,x)$. It is well known that the cyclic subgroup $\langle t\rangle$ is normal in $\pi_1(M,x)$ and is center when $M$ and $B_M$ are both orientable. The Seifert fibration $p$ induced a quotient homomorphism $p_*:\pi_1(M,x)\to \pi_1(M,x)/\langle t\rangle\cong \pi_1(B_M)$ on fundamental groups (see \cite[Chapter II, \S4]{JS} or \cite[Chapter 12]{He}).

\begin{prop}\label{main result of inenssential case}
Let $M$ be a compact orientable Seifert manifold with hyperbolic orbifold $B_M$, and $f:M\rightarrow M$ be an orientation-reversing homeomorphism. Suppose $f_w: \pi_1(M,x)\rightarrow \pi_1(M,x)$ is the automorphism induced by $f$, where $w$ is an $f$-route with $x=w(0)$. If the fixed point class $\F_w$ corresponding to $w$ is inessential, then $\rank \fix(f_w)\leq 3.$
\end{prop}

To prove Proposition \ref{main result of inenssential case}, we need the following useful lemma which is from \cite[Theorem 3.11]{JWW}.

\begin{lem}\label{fiber-preserving homeomorphism of Seifer mfd}
Suppose $M$ is a compact orientable Seifert manifold and $p:
M\rightarrow B_M$ is a Seifert fibration with hyperbolic orbifold
$B_M$. Then any homeomorphism on M is isotopic to a fiber-preserving homeomorphism with respect to this fibration.
\end{lem}

\begin{proof}[Proof of Proposition \ref{main result of inenssential case}]

By Lemma \ref{fiber-preserving homeomorphism of Seifer mfd}, $f$ can be isotopic to a fiber-preserving homeomorphism. Since the index $\ind(f,\F_w)$ and rank $\rank \fix(f_w)$ of fixed point class are homotopy invariants (see Fact (homotopy invariance) in Section 2), we can assume that $f$ is fiber-preserving and the base point $x$ is in a regular fiber $\tau$ in the following.

Note that $B_M$ is hyperbolic, then there is a finite covering
$q: S\rightarrow B_M$ of orbifold such that $S$ is a compact orientable hyperbilic surface by \cite [Theorem 2.5]{S}. The pull-back of the
Seifert fibration $p: M\rightarrow B_M$ via $q$ gives a finite covering
manifold $q': \widetilde M\rightarrow M$ with fibration
$p':\widetilde M\rightarrow B_{\widetilde M}=S$. After passing to further finite covering if necessary, we may assume that $q'$ is
a characteristic covering with finite sheets $d$. Since $S$ is an orientable surface, the fibration $p'$ is an orientable
circle bundle over the orientable surface $S$. So there is the following commutative diagram:
$$\xymatrix{
\widetilde M\ar[d]_{p'}\ar[r]^{q'}& M\ar[d]^p\\
S\ar[r]^q &B_M
}$$

Pick a point $\tilde x\in q'^{-1}(x)$ and a lifting $\tilde w$ of $w$ with $\tilde w(0)=\tilde x$. Since $q'$ is a finite characteristic covering, following from Lemma \ref{lifting of route}, there is a lifting $\tilde f$ of $f$ such that $\tilde w$ is an $\tilde f$-route. Let $\tilde f': S\to S$ denote the induced homeomorphism of $\tilde f$ on the orbifold $S$. Therefore, there is the following commutative diagram on fundamental groups:
$$\xymatrix{
& \pi_1(S,p'(x)) \ar[rr]^{\tilde f'_{p'\comp\tilde w}}\ar'[d][dd]_<<<{q_*}& &\pi_1(S,p'(x))\ar[dd]_{q_*}\\
\pi_1(\widetilde M,\tilde x)\ar[rr]^>>>>>>>>>>>{\tilde f_{\tilde w}}\ar[ur]^{p'_*}\ar[dd]_{q'_*}& &\pi_1(\widetilde M,\tilde x)\ar[ur]^{p'_*}\ar[dd]_<<<<<<{q'_*}\\
& \pi_1(B_M) \ar'[r][rr]& &\pi_1(B_M)\\
\pi_1(M,x) \ar[ur]^{p_*}\ar[rr]^{f_w}& &\pi_1(M,x)\ar[ur]^{p_*}
}$$

Note that $f:M\to M$ is a fiber-preserving homeomorphism that reverse the orientation of $M$, then the lifting $\tilde f: \widetilde M\to \widetilde M$ is a fiber-preserving homeomorphism that reverse the orientation of $\widetilde M$. Moreover, since the $f$-fixed point class $\F_w$ is inessential, the $\tilde f$-fixed point class $\F_{\tilde w}$ is also inessential by Lemma \ref{fixed point class lifting}. Therefore, $\fix(\tilde f'_{p'\comp\tilde w})$ is trivial or the free cyclic group $\Z$ according to conclusion (1) of Proposition \ref{fixed subgp of inessential fpc on circle bundles}. We claim that

\begin{claim}\label{claim}
Let $H=p_*\fix(f_w)\subseteq \pi_1(B_M)$ and $H^d=\{h^d|h\in H\}$. Then
$$H^d\subseteq q_*\fix(\tilde f'_{p'\comp\tilde w})\leq\pi_1(B_M)$$
where $q_*\fix(\tilde f'_{p'\comp\tilde w})$ is trivial or the free cyclic group $\Z$.
\end{claim}

\noindent\emph{Proof.}
Let $\alpha\in \fix(f_w)\leq \pi_1(M,x)$. Recall that $q':\widetilde M\to M$ is a characteristic covering of sheets $d$, then
$$q'_*:\pi_1(\widetilde M,\tilde x)\to \pi_1(M,x)$$
is an injective homomorphism and the image $q'_*\pi_1(\widetilde M,\tilde x)$ is a characteristic subgroup of $\pi_1(M,x)$ with index $d$.
Hence the power
$$\alpha^d\in q'_*\pi_1(\widetilde M,\tilde x)\leq\pi_1(M,x),$$
furthermore, following from the commutative diagram, we have
$$\tilde f_{\tilde w}q'^{-1}_*(\alpha^d)=q'^{-1}_*f_w(\alpha^d)=q'^{-1}_*(\alpha^d),$$
and
$$\tilde f'_{p'\comp\tilde w}p'_*(q'^{-1}_*(\alpha^d))=p'_*\tilde f_{\tilde w}(q'^{-1}_*(\alpha^d))=p'_*q'^{-1}_*(\alpha^d).$$
Namely,
$$p'_*q'^{-1}_*(\alpha^d)\in \fix(\tilde f'_{p'\comp\tilde w}).$$
Then
$$(p_*(\alpha))^d=p_*(\alpha^d)=q_*p'_*q'^{-1}_*(\alpha^d)\in q_*\fix(\tilde f'_{p'\comp\tilde w})\leq\pi_1(B_M).$$
Hence
$$H^d\subseteq q_*\fix(\tilde f'_{p'\comp\tilde w})\leq\pi_1(B_M).$$
Note that $q_*$ is induced by the covering $q$, then $q_*$ is an injective homomorphism, and $q_*\fix(\tilde f'_{p'\comp\tilde w})$ is isomorphic to $\fix(\tilde f'_{p'\comp\tilde w})$ which is trivial or the free cyclic group $\Z$. Hence the Claim holds.

Now we will finish the proof of Proposition \ref{main result of inenssential case}.

By Proposition \ref{roots of cyclic subgp of Fuchsian gps} and Claim \ref{claim}, $H=p_*\fix(f_w)$ is a metacyclic group. Since
$$p_*:\pi_1(M,x)\to \pi_1(M,x)/\langle t\rangle =\pi_1(B_M)$$
is the quotient map, we have $\fix(f_w)\leq p_*^{-1}(H)$ which is an extension of the metacyclic group $H$ by the infinite cyclic group $\langle t\rangle$, namely, $\fix(f_w)$ is a subgroup of the group $p_*^{-1}(H)$ which has a cyclic normal subgroup $\langle t\rangle$ such that the quotient group $H$ is metacyclic. So by Lemma \ref{subgp of metacyclic gp}, $\rank\fix(f_w)\leq 3$.
\end{proof}

Now we give the proof of Theorem \ref{main thm}. Firstly, we have the following lemma on the ranks of Seifert manifold groups.

\begin{lem}\label{rank of Seifert mfd}
Suppose $M$ is a compact orientable Seifert manifold with hyperbolic orbifold $B_M$. Then $\rank\pi_1(M)\geq 2.$
\end{lem}

\begin{proof}
It is clear from \cite[Proposition 4.3]{Z} which gives a description on the ranks of Seifert manifold groups.
\end{proof}

\begin{proof}[Proof of Theorem \ref{main thm}]
Let $f_\pi=f_w: \pi_1(M,x)\rightarrow \pi_1(M,x)$ be induced by some $f$-route $w$ with $x=w(0)$.

If $\F_w$ is essential, then $\rank \fix(f_w)<2\rank\pi_1(M)$ by Proposition \ref{rank of essential fpc}.

If $\F_w$ is inessential, then $\rank \fix(f_w)\leq 3$ by Proposition \ref{main result of inenssential case}. Recall that the orbifold $B_M$ of $M$ is hyperbolic, then $\rank \pi_1(M)\geq 2$ according to Lemma \ref{rank of Seifert mfd}. Therefore, we have
$$\rank \fix(f_w)<2\rank \pi_1(M).$$
The proof is finished.
\end{proof}

\section{Examples and Questuins}

Now we give some examples and questions.

\begin{exam}
Let $S_n$ be a closed orientable surface of genus $n\geq 2$. Define an orientation-reversing homeomorphism $f$ as follows:
$$f=f_1\times f_2:~ S_n\times S^1\to S_n\times S^1,$$
where $f_1: S_n\to S_n$ is a reflection on a simple closed curve $\gamma$, and $f_2: S^1\to S^1$ is a rotation. Then all the fixed point classes of $f$ are inessential, and $f$ induces an automorphism $f_{\pi}$ of $\pi_1(S_n\times S^1)$ such that
$$\fix(f_{\pi})=\pi_1(\gamma\times S^1)\cong \Z\oplus\Z.$$
Namely, there is an inessential fixed point class which has $\rank\fix(f_{\pi})=2$.
\end{exam}

\begin{ques}
Is there an orientation-reversing homeomorphism $f$ of a Seifert manifold $M$ whose inessential fixed point class has $\rank\fix(f_{\pi})=3$? Namely, is the bound $3$ in Proposition \ref{main result of inenssential case} sharp?
\end{ques}





\end{document}